\documentclass[12pt]{amsart}

\allowdisplaybreaks[1]
\usepackage{enumerate}
\usepackage{graphicx}
\usepackage{allrunes}
\usepackage{amsmath}
\usepackage{amssymb}
\usepackage{amsfonts}
\usepackage[dvipsnames]{xcolor}
\usepackage{hyperref}
\allowdisplaybreaks
\makeindex

 \newtheorem{theorem}{Theorem}[section]
 \newtheorem{corollary}[theorem]{Corollary}
 
 \newtheorem{proposition}[theorem]{Proposition}

\theoremstyle{definition}

\theoremstyle{remark}

\newtheorem{fact*}{Fact}

\DeclareMathOperator\im{\mathrm {Im~}}

\renewcommand{\D}{\mathbb{D}}
\newcommand{\C}{\mathbb{C}}

\renewcommand{\R}{\mathbb{R}}

\newcommand{\cc}[1]{\overline{#1}}

\newcommand{\norm}[1]{\left\Vert#1\right\Vert}

\newcommand{\til}{\raise.17ex\hbox{$\scriptstyle\mathtt{\sim}$}}

\newcommand{\ph}{\varphi}

\newcommand\ep{\varepsilon}

\newcommand\beq{\begin{equation}}

\newcommand\eeq{\end{equation}}

\newcommand{\bbm}{\left[ \begin{smallmatrix}}
\newcommand{\ebm}{\end{smallmatrix} \right]}
\newcommand{\bpm}{\left( \begin{smallmatrix}}
\newcommand{\epm}{\end{smallmatrix} \right)}
\numberwithin{equation}{section}

\newlength{\Mheight}
\newlength{\cwidth}

\newcommand{\dfn}[1]{{\bf #1}\index{#1}}

\newcommand{\GG}{\textarc{f}}
\newcommand{\sov}{\textarc{m}}

\title[Automatic real analyticity]{Automatic real analyticity and a regal proof of a commutative multivariate L\"owner theorem}
\author[J. E. Pascoe]{
J. E. Pascoe
}
\address{Department of Mathematics\\
1400 Stadium Rd\\
  University of Florida\\
 Gainesville, FL 32611}
\email[J. E. Pascoe]{pascoej@ufl.edu}

\author[R. Tully-Doyle]{
Ryan Tully-Doyle
}
\address{Department of Mathematics and Physics \\
University of New Haven\\
West Haven, CT 06516 }
\email[R. Tully-Doyle]{rtullydoyle@newhaven.edu}

\date{\today}

\begin{document}

\begin{abstract}
We adapt the ``royal road" method used to simplify automatic analyticity theorems in noncommutative function theory to several complex variables.
We show that certain families of functions must be real analytic if they have certain nice properties on one dimensional slices.
Let $E \subset \R^d$ be open.
A function $f:E \to \R$ is \emph{matrix monotone lite} if $f(\ph_1(t), \ldots, \ph_d(t))$ is a matrix monotone function of $t$ whenever $t \in (0,1)$, the $\ph_i$ are automorphisms of the upper half plane, and the tuple $(\ph_1(t), \ldots, \ph_d(t))$ maps $(0,1)$ into $E$.
We use the ``royal road" to show that a function is matrix monotone lite if and only if it analytically continues to the multi-variate upper half plane as a map into the upper half plane.
Moreover, matrix monotone lite functions in two variables are locally matrix monotone in the sense of Agler-McCarthy-Young.
\end{abstract}

\maketitle

\section{Introduction}


The goal of this manuscript is to transform the ingredients from the so-called ``royal road'' proof of the noncommutative multivariate L\"owner theorem \cite{ptdroyal} and chop away all of the excess material to leave the core content of a deconstructed royal road theorem suitable for complex variables.
The royal road gives a way to lift automatic analyticity theorems, such as L\"owner's theorem, which will be described momentarily, from one variable to many variables. We apply the royal road machinery, as an example of our technique, to prove a 
commutative L\"owner type theorem, and in the process give a new proof of the relaxed Agler-McCarthy-Young theorem \cite{amyloew, pascoelownote} in two variables, while in several variables giving an entirely new theorem.  Moreover, this technique is based on complex variables and classical analysis, rather than the technically daunting Agler model machinery used in \cite{amyloew}, and thus may be significantly more flexible.

Let $f:(a,b) \to \R$ be a function. We say that $f$ is \dfn{matrix monotone} if 
\[
A \leq B \Rightarrow f(A) \leq f(B)
\]
for all $A, B$ self-adjoint of the same size with spectrum in $(a,b)$, where $A \leq B$ means that $B - A$ is positive semidefinite, and the function $f$ is evaluated via the matrix functional calculus.  The seemingly innocuous condition of matrix monotonicity is in fact very rigid, as demonstrated by K. L\"owner \cite{lo34}.

\begin{theorem}[L\"owner 1934 \cite{lo34}]\label{low}
Let $f:(a,b) \to \R$. A function $f$ is matrix monotone if and only if $f$ is real analytic on $(a,b)$ and analytically continues to the upper half plane in $\C$ as a map into the closed upper half plane.
\end{theorem}

One natural direction for generalizing L\"owner's theorem is to ask when a matrix monotone function in $d$ variables on some real domain possesses an analytic continuation to a complex poly upper half plane $\Pi^d = \{(z_1, \ldots, z_d)\in \C^d| \im z_i >0\}$. 

Let $E \subset \R^d$ be an open set. For tuples $S = (S_1, \ldots, S_d), T = (T_1, \ldots, T_d)$,
\[
S \leq T \text{ if } S_r \leq T_r, r = 1, \ldots, d.
\]
Let $CSAM_n^d(E)$ be the set of $d$-tuples of commuting self-adjoint matrices of size $n$ with joint spectrum in $E$.

A function $f: E \to \R$ is \dfn{locally matrix monotone} if for every $n$, for every $C^1$ curve $\gamma: [0,1] \to CSAM_n^d(E)$ with $\gamma_i'(t) \geq 0$ for all $i$ and all $t \in [0,1]$, 
\[
t_1 \leq t_2 \Rightarrow f(\gamma(t_1)) \leq f(\gamma(t_2)).
\]

The following generalization of L\"owner's theorem in this setting is due to Agler, McCarthy, and Young \cite{amyloew} with a mild refinement in \cite{pascoelownote}.

\begin{theorem}[Agler, McCarthy, Young \cite{amyloew}]
A function $f: E \to \R$ is locally matrix monotone if and only if $f$ analytically continues to $\Pi^d$ as map $f: E \cup \Pi^d \to \cc{\Pi}$ in the Pick-Agler class.
\end{theorem}

The Pick-Agler class is a class of complex analytic functions that are of technical importance in operator theory.
In the case where $d = 1$ or $d = 2$, the Pick-Agler class is equivalent to the set of all analytic functions $f:\Pi^d \to \cc\Pi$ by Ando's inequality \cite{and63}. For $d>2,$ the
Pick-Agler class is a strict subset, as Ando's inequality fails \cite{par70, var74}.

A function $f:E \to \R$ is \dfn{matrix monotone lite} if $f(\ph_1(t), \ldots, \ph_d(t))$ is a matrix monotone function of $t$ whenever $t \in (0,1)$, the $\ph_i$ are automorphisms of the upper half plane, and the tuple $(\ph_1(t), \ldots, \ph_d(t))$ maps $(0,1)$ into $E$. A locally matrix monotone function is matrix monotone lite. 
We prove the following theorem.
\begin{theorem}[Diet L\"owner theorem]\label{lownerfull}
Let $E$ be a connected open set. A function $f: E \to \R$ is matrix monotone lite if and only if $f$ analytically continues to $\Pi^d$ as map $f: E \cup \Pi^d \to \cc{\Pi}$.
\end{theorem}

Theorem \ref{lownerfull} follows directly from Proposition \ref{lownerorthant} as the definition of matrix monotone lite is conformally invariant.

Note that our theorem solves the mystery of how the Pick functions in $d$ variables which have a real-valued continuation to some $E \subseteq \R^d$ relate to matrix monotonicity; they are exactly the matrix monotone lite functions. Moreover, we have the following shocking corollary.

\begin{corollary}
In two variables, a function is locally matrix monotone if and only if it is matrix monotone lite.
\end{corollary}

It seems likely that there should be a proof that once we know that a function continues to the upper half plane $\Pi^d$ that it must be in the Pick-Agler class whenever it is locally matrix monotone. However, since this would involve a significant incursion into operator theory, we eschew this endeavour.

\section{Automatic analyticity}

A \dfn{fiefdom} is a set of domains, denoted by $\GG$, in $\R^d$ satisfying:
\begin{description}
\item[Translation invariance] For all $G \in \GG$ and $r \in \R^d$, $G + r \in \GG$.
\item[Scale invariance] If $t > 0$ and $G \in \GG$, $tG \in \GG$.
\item[Closure under intersection] For all $G, H \in \GG$, $G \cap H \in \GG$.
\item[Locality] For any $x \in \R^d$ and $\ep > 0$, $B_\R(x, \ep) \in \GG$.
\end{description}

The class of convex sets in $\R^d$ is an example of a fiefdom, as is the class of all open sets in $\R^d$.

A \dfn{noble class} $\sov$ is a set of functions on domains contained in a fiefdom $\GG$ satisfying:
\begin{description}
\item[Functions] For all $G \in \GG$, $\sov(G)$ is a set of locally bounded measurable functions.
\item[Closure under localization] If $f \in \sov(G)$ and $H \subseteq G$ then $f|_H \in \sov(H)$.
\item[Closure under convolution] The set of functions $\sov(G)$ is convex and closed under pointwise weak limits.
\item[One variable knowledge] If $a_i \leq b_i$ for each $i,$ then
\[
f_{\cc{ab}}(t) := f\left(\frac{1-t}{2} a + \frac{1 +t}{2} b \right)
\]
analytically continues to $\D$ as a function of $t$.
\item[Control] There is a map $\gamma$ taking each pair $(x,f)$ to a non-negative number satisfying 
\begin{enumerate}
\item For each $\ep > 0$ there is a universal constant $c(\ep)$ such that $\inf_{X \in B_\R(x, \ep)} \gamma(x,f) \leq c(\ep)\norm{f}_{B_\R(x, \ep)}$.
\item There is a universal positive valued function $e$ on $\R^+$ satisfying the following. Write $f_{\cc{ab}}(t) = \sum a_n t^n$. Then, 
\[
\norm{a_n} \leq \gamma(x, f) e(\norm{b - a}).
\]
\item If $H \subseteq G$ and $x \in H$ then $\gamma(x, f|_H) \leq \gamma(x, f)$.
\end{enumerate} 
\end{description}

These definitions are drawn from the notions of \emph{dominions} and \emph{sovereign classes} established in the noncommutative setting in \cite{ptdroyal}.

The set of matrix monotone lite functions in $d$ variables on the fiefdom of open domains is a noble class equipped with the control function $\gamma(x, f) = \norm{f(x)} + \norm{Df(x)}$. This is essentially an exercise using the Nevanlinna representation, the details of which are given in \cite[Proposition 3.1]{ptdroyal}. We give a proof here for completeness.

\begin{proposition}
	The set of matrix monotone lite functions in $d$ variables is a noble class on the fiefdom of open domains with the control function $\gamma(x, f) = \norm{f(x)} + \norm{Df(x)}.$

	Here, $\|Df(X)\| = \sup_{h\in (0,\infty)^d} \frac{|Df(X)[h]|}{\|h\|}.$
\end{proposition}
\begin{proof}
	The properties of closure under localization and convolution are clear from the definition. One variable knowledge comes from the fact that when restricted to $1$-dimensional positively-oriented slices,
	$f$ is a classical one-variable matrix monotone function. The function $f$ is monotone on $\mathbb{R}^d$ and thus locally bounded.
 	
	Now, we show that the control properties hold.
	Fix $\ep >0$. Suppose that $B_{\R}(x, \ep)$ is contained in the domain of $f$. Without loss of generality, $0 = x$.
	Fix $h$ in $B_{\R}(0,\ep)$ such that each $h_i \geq 0.$
	Note $f(th)$ is a matrix monotone function in one variable $t$ on $(-1,1)$ and therefore analytically continues to a self-map of the upper half plane and is subject to classical integral representations.	
	So $f(zh)$ has a Nevanlinna representation \cite{nev22} given by
	\begin{align*}
	f(zh) &= a_0 + \int_{[-1,1]} \frac{z}{tz + 1} \, d\mu(t) \\
	&= a_0 + z \sum_{i=0}^\infty \int t^i z^i \, d\mu(t).
	\end{align*}
	Note that this shows that $|a_n| = \int |t|^{n-1}  \, d\mu(t)\leq \int  \, d\mu(t)  = a_1 $ for $n\geq 1.$
	Moreover,
	\[
	f(zh) - f(-zh) = 2 z \sum_{i=0}^\infty \int z^{2i} t^{2i} \, d\mu.
	\]
	This shows that
	\[
	\norm{Df(0)[h]} \leq \norm{f}_{B_\R(0,\ep)}. 
	\]
	Therefore, 
	\[
	\norm{Df(0)[h]} \leq \frac{1}{\ep} \norm{f}_{B_\R(0,\ep)},
	\]
	which is bounded by $(1 + \frac{1}{\ep}) \norm{f}_{B_\R(0, \ep)}$.

\end{proof}

We will need the following quantitative wedge-of-the-edge theorem, which appeared as Corollary 2.3 in \cite{ptdroyal}, but is essentially previous work in \cite{pascoecmb, pascoeblmswedge}.

\begin{theorem}[The quantitative wedge-of-the-edge theorem]\label{wedge}
There are universal constants $\delta, \ep >0$  satisfying the following.
Let $h_d$ be a sequence of generalized homogeneous polynomials of degree $d$ such that
$\sum \|h_d(x)\|$ is bounded by $1$ on $[0,1]^d.$
The formula $\sum h_d(z)$ defines an analytic function on $B_{\C}(0,\delta)$ which is bounded by $\ep.$
\end{theorem}

We now prove the deconstructed royal road theorem.

\begin{theorem}[The deconstructed royal road theorem]\label{deconstruct}
Any function in a noble class is real analytic.
\end{theorem}

\begin{proof}
We will show that a function in a noble class defined on $B_\R(0,2)$ which is bounded by $1$ is real analytic at $0$. Let $\ph$ be a compactly supported positive smooth function on $\R^d$. Define 
\[
\ph_\alpha(x) := \frac{1}{\alpha}\ph\left(\frac{1}{\alpha}x\right).
\]
Consider $f_\alpha = \ph_\alpha * f$. Note that by closure under convolution that $f_\alpha|_{B_\R(0,2 - \ep)}$ will be in our noble class for small enough $\alpha$. Now choose $y \in B_\R(0, \frac{\delta}{2})$ such that $\gamma(y, f) \leq 2c(\delta) \norm{f}_{B_\R(0, \frac{\delta}{2})}$, where $\delta$ is the constant in Theorem \ref{wedge}. Note that $f_\alpha$ is smooth at $y$, and by one variable knowledge we have that $f_\alpha(y + x) = \sum h_n(x)$ for $x \in [0,1]^d$. By Theorem \ref{wedge}, $f_\alpha$ analytically continues to $B_\R(y,\delta)$ and is bounded by $\ep$ there. Therefore, $f_\alpha$ is analytic on $B_\R(0,\frac{\delta}{2})$ and bounded by $\ep$. Taking $\alpha$ to $0$ implies that $f$ analytically continues and is bounded by $\ep$ there by a normal families argument.
\end{proof}

\section{L\"owner's theorem}
We now prove L\"owner's theorem for the positive orthant, which by conformal invariance implies Theorem \ref{lownerfull}.
\begin{proposition}\label{lownerorthant}
Let $f:(0,\infty)^d \to \R$. The function $f$ is matrix monotone lite if and only $f$ analytically continues to the upper half-plane $\Pi^d$. 
\end{proposition}

\begin{proof}
The converse direction follows directly from L\"owner's theorem in one variable.

Suppose that $f$ is matrix monotone lite. By the deconstructed royal road Theorem \ref{deconstruct}, $f$ is real analytic on $(0,\infty)^d$. Additionally, for each $x \in \R^d$ and $y \in (0,\infty)^d$, the induced function $g_{x,y}(w) = f(x + yw)$ is matrix monotone and thus has a well-defined extension to all $w$ in the upper half plane $\Pi.$ Any point $z$ in the upper half plane $\Pi^d$ can be written in the form $x + yw$. Moreover, if $x_1 + y_1w_1 = x_2 +y_2w_2,$
then $g_{x_1,y_1}(w_1) = g_{x_2,y_2}(w_2)$ by the real analyticity of $f$ on $(0,\infty)^d$ and the fact that the relation holds when $x_1 + y_1w_1, x_2 +y_2w_2 \in (0,\infty)^d.$ Therefore, we have a well-defined 
extension of $f$ to $\Pi^d$ via the formula $f(x+yw)= g_{x,y}(w).$

It remains to show that the extension is analytic. Note that, for $z \in \Pi^d$, for each $y \in (0,\infty)^d$ such that there exists an $x, w$ such that $z=x+yw,$ we have that the directional derivative of $f$ in the direction
$y$ exists. Note that the set of such $y$ has nonempty interior, and therefore the derivative exists in a full basis set of directions. Thus, by Hartog's theorem, $f$ is analytic, as each of the directional derivatives exist in some coordinate system. 
\end{proof}

\bibliography{references}
\bibliographystyle{plain}

\end{document}